\documentclass[12pt]{article}
\usepackage{amsmath}
\usepackage{amssymb}
\usepackage{amsthm}
\usepackage{mathrsfs}
\usepackage{latexsym}

\usepackage{amsthm}
\usepackage{graphicx}
\usepackage{setspace}
\usepackage{xy}
\input xy
\xyoption{all}
\theoremstyle{plain}
\newtheorem{theorem}{Theorem}[section]
\newtheorem{corollary}[theorem]{Corollary}
\newtheorem{proposition}[theorem]{Proposition}
\newtheorem{lemma}[theorem]{Lemma}

\theoremstyle{remark}
\newtheorem{remark}[theorem]{Remark}
\theoremstyle{definition}
\newtheorem{definition}[theorem]{Definition}
\newtheorem{construction}[theorem]{Construction}

\newtheorem{condition}[theorem]{Condition}

\newtheorem{proposition-definition}[theorem]{Proposition-Definition}

\newtheorem{convention}[theorem]{Convention}

\newcommand{\C}{\mathbb{C}}

\newcommand{\G}{\mathbb{G}}
\newcommand{\pp}{\mathbb{P}}

\newcommand{\Z}{\mathbb{Z}}
\newcommand{\Q}{\mathbb{Q}}
\newcommand{\vv}{\mathrm{virt}}

\title{Virtual push-forwards}
\author{Cristina Manolache}
\date{}

\begin{document}
\maketitle
\begin{abstract}Let $p:F\to G$ be a morphism of stacks of positive \emph{virtual} relative dimension $k$ and let $\gamma\in H^k(F)$. We give sufficient conditions for $p_*\gamma\cdot[F]^{\vv}$ to be a multiple of $[G]^{\vv}$. We apply this result to show an analogue of the conservation of number for virtually smooth families. We show implications to Gromov-Witten invariants and give a new proof of a theorem in \cite{mop} which compares the virtual classes of moduli spaces of stable maps and moduli spaces of stable quotients.
\end{abstract}

\tableofcontents
\section {Introduction}
Virtual fundamental classes have been introduced by Li-Tian \cite{lg} and Behrend-Fantechi \cite{bf} and in the past fifteen years have become a useful tool when one has to deal with badly-behaved (i.e singular, with several components of possibly different dimensions) moduli spaces. One of the main problems when working with virtual fundamental classes is that in certain situations they fail to behave as fundamental classes do. One easy example is the following. Let  $f:F\to G$ be a finite morphism of stacks and suppose that $F$ and $G$ have pure dimension. Let $G_1,...,G_s$ denote the irreducible components of $G$. Then we have that
\begin{equation*}
f_*[F]=n_1[G_1]+...+n_s[G_s]
\end{equation*}
for some $n_1,...,n_k\in\Q$. On the contrary, given a morphism $f:F\to G$ of stacks which possess virtual classes of the same virtual dimension, we have no reasons to believe that the following relation holds
\begin{equation*}
f_*[F]^{\vv}=n_1[G_1]+...+n_s[G_s]
\end{equation*}
where $G_1,...,G_k$ are cycles on $G$ such that $[G]^{\vv}=[G_1]+...+[G_s]$.
\\In this article we find sufficient conditions for the above condition to hold. More generally, let $f:F\to G$ be a morphism of stacks which possess virtual classes of dimension $k_1$ respectively $k_2$ such that $k:=k_1-k_2\geq0$ and let $\gamma\in H^{k}$. The main result states that if the relative obstruction theory of $f$ is perfect, and $G$ is connected, then the push-forward of $\gamma\cdot [F]^{\vv}$ along $f$ is equal to a scalar multiple of the virtual class of $G$.
\\As applications we show that given a virtually smooth family $f:F\to G$, we have an analogous statement to the conservation of number principle. This shows that the virtual Euler characteristic is locally constant in virtually smooth families.
\\We also show that for any smooth fibration $p:X\to\pp^r$, we have that for the the induced morphism $\bar{p}: \overline{M}_{g,n}(X,\beta)\to\overline{M}_{g,n}(\pp^r,p_*\beta)$ and any $\gamma\in H^{k}( \overline{M}_{g,n}(X,\beta))$ we have that  $\bar{p}_*\gamma\cdot[\overline{M}_{g,n}(X,\beta)]^{\vv}=n[\overline{M}_{g,n}(\pp^r,p_*\beta)]^{\vv}$, for some $n\in\Q$.
\\
\\This problem has already been studied by A. Gathmann for the natural map between the moduli space of stable maps to a projective bunlde $p:\pp_X(L\oplus\mathcal{O})\to X$ and the moduli space of stable maps to $X$. His computation uses localization and it is rather long. I was also influenced by work of B. Kim \cite{kim}, who compared a certain intersection product on the moduli space of stable maps to a projective bundle over a variety $X$ to the virtual class of the moduli space of stable maps to $X$.
\\Our approach is similar to the one of H-H Lai, who analyzes the map between the moduli spaces of stable maps to a certain blow-up and the moduli space of stable maps to the base variety. Our Lemma \ref{vpf} is a reformulation of arguments present in \cite{h}.
\\In \cite{mop} Marian, Oprea and Pandharipande have constructed a new compactification of the space of genus $g$ curves in Grassmannians which admits a virtual class. In the special case of curves in projective spaces there exists a morphism between the moduli space of genus $g$ stable maps of degree $d$ to a projective space $\pp^n$ and the moduli space of rank $n-1$ quotients (see section \ref{mapstoquot}). We give a different proof to theorem 3 in \cite{mop} which compares the two virtual classes.

\paragraph{Acknowledgements} I owe the main statement of this paper to Barbara Fantechi and Angelo Vistoli who had suggested me to work with homology groups, rather than Chow groups. I am very grateful to Ionu\c{t} Ciocan-Fontanine, Gavril Farkas, Carel Faber and Y.P. Lee for several useful discussions and suggestions. I would also like to thank Ionu\c{t} for making the manuscript \cite{kim} available to me.
\section{Background}
\paragraph{Notation and conventions.}
We take the ground field to be $\C$.
\\An Artin stack is an algebraic stack in the sense of \cite{lau} of finite type over the ground field.
\\Unless otherwise specified we will try to respect the following convention: we will usually denote schemes by $X,\ Y,\ Z$, etc, Deligne-Mumford stacks by $F,\ G,\ H$, etc. and Artin stacks for which we know that they are not Deligne-Mumford stacks (such as the moduli space of genus-$g$ curves or vector bundle stacks) by gothic letters $\mathfrak{M}_{g},\ \mathfrak{E},\ \mathfrak{F}$, etc. 
\\By a commutative diagram of stacks we mean a 2-commutative diagram of stacks and by a cartesian diagram of stacks we mean a 2-cartesian diagram of stacks.
\\Chow groups for schemes are defined in the sense of \cite{f}; this definition has been extended to DM stacks (with $\Q$-coefficients) by Vistoli (\cite{v}) and to algebraic stacks (with $\Z$-coefficients) by Kresch (\cite{k}). We will consider Chow groups (of schemes/stacks) with $\mathbb{Q}$-coefficients.
\\By Homology we mean Borel-Moore homology (see the Appendix in \cite{tom} for a definition for stacks).
\\For a fixed stack $F$ we denote by $\mathcal{D}_F$ the derived category of coherent $\mathcal{O}_F$ modules.
\\For a fixed stack $F$ we denote by $L_F$ its cotangent complex defined in \cite{ol}.

\subsection{Obstruction Theories}

\begin{definition}
 Let $E^{\bullet}\in\mathcal{D}^{\leq0}_F$. $E^{\bullet}$ is said to be of perfect amplitude if there exists $n\geq0$ such that $E^{\bullet}$ is locally isomorphic to $[E^{-n}\to...\to E^{0}]$, where $\forall i\in\{-n,...,0\}$, $E^i$ is a locally free sheaf.
\end{definition}

\begin{definition}Let $E^{\bullet}\in\mathcal{D}^{\leq0}_X$. Then a homomorphism $\Phi: E^{\bullet}\to L^{\bullet}$ in $\mathcal{D}_F$ is called an obstruction theory if $h^0(\Phi)$ is an isomorphism and $h^{-1}(\Phi)$ is surjective. If moreover, $E^{\bullet}$ is of perfect amplitude, then $E^{\bullet}$ is called a perfect obstruction theory.
\end{definition}
\begin{convention}Unless otherwise stated by a perfect obstruction theory we will always mean of perfect amplitude  contained in $[-1,0]$.
\end{convention}

\subsection{Cone stacks}
\begin{definition}
Let $X$ be a scheme and $\mathcal{F}$ be a coherent sheaf on $X$. We call $C(\mathcal{F}):=Spec Sym(\mathcal{F})$ an abelian cone over $X$.
\end{definition}
As described in \cite{bf}, Section 1, every abelian cone $C(\mathcal{F})$ has a section $0:X\to C(\mathcal{F})$ and an $\mathbb{A}^1$-action.
\begin{definition} An $\mathbb{A}^1$-invariant subscheme of $C(\mathcal{F})$ that contains the zero section is called a cone over $X$.
\end{definition}
Similarly, Behrend and Fantechi define in \cite{bf} Section 1, abelian cone stacks and cone stacks. Let us recall the definition.
\begin{definition}
Let $F$ be a stack and let $[E^0\to E^1]$ be an element in $\mathcal{D}_F$. We call the stack quotient $[E^1/E^0])$ (in the sense of \cite{bf} Section 2) an abelian cone stack over stack $F$ .
\\A cone stack is a closed substack of an abelian cone stack invariant under the action of $\mathbb{A}^1$ and containing the zero section.
\end{definition}
\begin{convention}\label{stack}From now on, unless otherwise stated, by cones we will mean cone-stacks.
\end{convention}
\subsection{Virtual pull-backs in Chow groups}
In the following we recall the main results in \cite{eu}.
\begin{condition} We say that a morphism $F\to G$ of algebraic stacks and a vector bundle stack $\mathfrak{E}\to F$ satisfy condition ($\star$) if we have fixed a closed embedding $C_{f}\hookrightarrow \mathfrak{E}$.

\end{condition}
\begin{convention}
Will shortly say that the pair $(f,\mathfrak{E})$ satisfies condition ($\star$).
\end{convention}

\begin{remark}
Let us consider a Cartesian diagram
\begin{equation*}
 \xymatrix {F^{\prime} \ar[r]\ar[d]_{p} & G^{\prime}\ar[d]^q\\
             F\ar[r]^f & G.}
\end{equation*}
If $\mathfrak{E}$ is a vector bundle on $F$ such that $C_{F/G}\hookrightarrow \mathfrak{E}$ is a closed embedding, then $C_{F^{\prime}/G^{\prime}}\hookrightarrow p^*\mathfrak{E}$ is a closed embedding.
\end{remark}
\begin{construction}
Let $F$ be a DM stack and $\mathfrak{E}$ a vector bundle stack of (virtual) rank $n$ on $F$ such that $(f,\mathfrak{E})$ that satisfies condition $(\star)$ for $f$, we construct a pull-back map $f^{!}_\mathfrak{E}:A_*(G)\to A_{*-n}(F)$ as the composition
\begin{equation*}
 A_*(G)\stackrel{\sigma}\rightarrow A_*(C_{F/G})\stackrel{i_*}{\rightarrow}A_*(\mathfrak{E})\stackrel{s^*}{\rightarrow} A_{*-n}(F),
\end{equation*}
where
\begin{enumerate}
\item $\sigma$ is defined on the level of cycles by $\sigma(\sum n_i[V_i])=\sum n_i [C_{V_i\times_GF/V_i}]$ 
\item $i_*$ is the push-forward via the closed immersion $i$ 
\item $s^*$ is the morphism of Proposition 5.3.2 in \cite{K:03}.
\end{enumerate}
The fact that $\sigma$ is well defined has been checked in \cite{eu}. 
 \\Going further, for any cartesian diagram
\begin{equation*}
 \xymatrix {F^{\prime} \ar[r]^{f^{\prime}}\ar[d]_{p} & G^{\prime}\ar[d]^q\\
             F\ar[r]^f & G}
\end{equation*}
such that $\mathfrak{E}\to F$ satisfies condition $(\star)$ for $f$, let $f^{!}_{\mathfrak{E}}:A_*(G^{\prime})\to A_{*-n}(F^{\prime})$ be the composition
\begin{equation*}
 A_*(G^{\prime})\stackrel{\sigma}\rightarrow A_*(C_{F^{\prime}/G^{\prime}})\stackrel{i_*}{\rightarrow} A_*(C_{F/G}\times_FF^{\prime})\stackrel{i_*}{\rightarrow}A_*(p^*\mathfrak{E})\stackrel{s^*}{\rightarrow}A_{*-n}(F^{\prime}).
\end{equation*}
\end{construction}
\begin{definition}\label{vp}
In the notation above, we call $f_{\mathfrak{E}}^!: A_*(G)\to A_*(F)$ a \textit{virtual pull-back}. When there is no risk of confusion we will omit the index.
\end{definition}
\begin{theorem} \label{relations}Consider a fibre diagram of DM stacks
\begin{equation*}
 \xymatrix {F^{\prime} \ar[r]^{f^{\prime}}\ar[d]_{q} & G^{\prime}\ar[d]^p\\
             F\ar[r]^f & G}
\end{equation*}
and let us assume that ${\mathfrak{E}}$ is a vector bundle stack of rank $d$ such that $(f,{\mathfrak{E}})$ satisfies condition $(\star)$ for $f$.
\\(i) (Push-forward) If $p$ is a proper morphism of DM-stacks and $\alpha\in A_k(G^{\prime})$, then $f^!_{\mathfrak{E}}p_*(\alpha)=q_*f^!_{\mathfrak{E}}\alpha$ in $A_{k-d}(F)$.\\
(ii)\ (Pull-back) If $p$ is flat of relative dimension $n$ and $\alpha\in A_k(G)$, then $f^!_{\mathfrak{E}}p^*(\alpha)=q^*f^!_{\mathfrak{E}}\alpha$ in $A_{k+n-d}(F^{\prime})$\\
(iii)(Compatibility) If $\alpha\in A_k(G^{\prime})$, then $f^{!}_{\mathfrak{E}}\alpha=f^{\prime!}_{g^*{\mathfrak{E}}}\alpha$ in $A_{k-d}(F^{\prime})$.
\end{theorem}
\begin{remark}
 As remarked before, the generalized Gysin pull-back is well-defined for smooth pull-backs. Let us show that the two definitions agree. By (i) above, it is enough to prove the claim for $\alpha=[G]$, for which it follows trivially by construction. 
\end{remark}

\begin{theorem}(Commutativity)
Consider a fiber diagram of Artin stacks
\begin{equation*}
\xymatrix{
F^{\prime}\ar[d]\ar[r]&G{^\prime}\ar[d]^g\\
F\ar[r]^f&G}
\end{equation*}such that $F^{\prime}$ and $G^{\prime}$ admit stratifications by global quotients. Let us assume $f$ and $g$ are morphisms of DM-type and let ${\mathfrak{E}}$ and ${\mathfrak{F}}$ be vector bundle stacks of rank $d$, respectively $e$ such that $(f,{\mathfrak{E}})$ and $(g,{\mathfrak{F}})$ satisfy condition $(\star)$. Then for all $\alpha\in A_k(G)$,
\begin{equation*}
 g^!_{\mathfrak{F}}f^!_{\mathfrak{E}}(\alpha)= f^!_{\mathfrak{E}}g^!_{\mathfrak{F}}(\alpha)
\end{equation*}
in $A_{k-d-e}(F^{\prime})$.
\end{theorem}
\begin{theorem}\label{bivar}
Let $F$ admit a stratification by global quotients, let $f:F\to G$ be a morphism and $\mathfrak{E}\to F$ be a rank-$n$ vector bundle stack on $F$ such that $(f,\mathfrak{E})$ satisfies  Condition ($\star$). Then $f_{\mathfrak{E}}^!$ defines a bivariant class in $A^{n}(X\to Y)$ in the sense of \cite{f}, Definition 17.1.
\end{theorem}

\begin{definition}\label{compatib}
 Let $F\stackrel{f}\rightarrow G\stackrel{g}\rightarrow \mathfrak{M}$ be DM-type morphisms of stacks. If we are given a distinguished triangle of relative obstruction theories which are perfect in $[-1,0]$ $$g^*E^{\bullet}_{G/\mathfrak{M}}\stackrel{\varphi}\rightarrow E^{\bullet}_{F/\mathfrak{M}}\to E^{\bullet}_{F/G}\to g^*E^{\bullet}_{G/\mathfrak{M}}[1]$$ with a morphism to the distinguished triangle $$g^*L_{G/\mathfrak{M}}\to L_{F/\mathfrak{M}}\to L_{F/G}\to g^*L_{G/\mathfrak{M}}[1],$$ then we call $(E^{\bullet}_{F/G},E^{\bullet}_{G/\mathfrak{M}},E^{\bullet}_{F/\mathfrak{M}})$ a compatible triple.
\end{definition}
\begin{theorem}\label{functoriality}
(Functoriality)
 Consider the composition
\begin{equation*}
 \xymatrix {F\ar[r]^f & G\ar[r]^g&\mathfrak{M}.}
\end{equation*}
Let us assume $f$, $g$ and $g\circ f$ have perfect relative obstruction theories $E^{\bullet}_{F/G}$, $E^{\bullet}_{G/\mathfrak{M}}$ and $E^{\bullet}_{F/\mathfrak{M}}$ respectively and let us denote the associated vector bundle stacks by $\mathfrak{E}_{F/G}$, $\mathfrak{E}_{G/\mathfrak{M}}$ and $\mathfrak{E}_{F/\mathfrak{M}}$ respectively. If $(E^{\bullet}_{F/G}, E^{\bullet}_{G/\mathfrak{M}},E^{\bullet}_{F/\mathfrak{M}})$ is a compatible triple, then for any $\alpha\in A_k(\mathfrak{M})$ $$(g\circ f)_{\mathfrak{E_{F/\mathfrak{M}}}}^!(\alpha)=f^!_{\mathfrak{E}_{F/G}}(g^!_{\mathfrak{E}_{G/\mathfrak{M}}}(\alpha)).$$
\end{theorem}

\section{Virtual pull-backs and algebraic equivalences}
Let $T$ be an irreducible smooth variety of positive dimension $m$. The notation $t:\{t\}\to T$ will be used to denote the inclusion of a closed point $t$ in $T$. If $p:\mathcal{F}\to T$ is given, then we denote by $F_t$ the stack $p^{-1}(t)$. Any $k+m$-cycle on $\mathcal{F}$ determines a family of $k$-cycle classes $\alpha_t\in A_k(Y_t)$, by the formula
\begin{equation*}
\alpha_t:=t^!\alpha
\end{equation*} 
If $f:\mathcal{F}\to\mathcal{G}$ is a morphism of stacks over $T$, we denote by
\begin{equation*}
f_t:F_t\to G_t
\end{equation*}
the induced morphism on the fibers over $t\in T$. 
\begin{proposition}
Let $f:\mathcal{F}\to\mathcal{G}$ be a morphism of stacks over $T$. If $E\to\mathcal{F}$ is a vector bundle-stack which satisfies condition ($\star$) for $f$, then $E_t$ satisfies condition  ($\star$) for $f_t$ and we have that
\begin{equation*}
 f_t^!(\alpha_t)=(f^!(\alpha))_t
\end{equation*}
in $A_*(f^{-1}(\alpha))$.
\begin{proof}
 The statement follows by the commutativity of virtual pull-backs.
\end{proof}
\end{proposition}
\begin{proposition}
Let $f:F\to G$ be a morphism of stacks and let $E\to F$ is a vector bundle-stack which satisfies condition ($\star$) for $f$. Then we have a morphism $f_E^!:B_*(G)\to B_{*-r}(F)$ which makes the diagram commute
\begin{equation*}
 \xymatrix{A_*(G)\ar[r]^{f_E^!}\ar[d]_{cl_G}&A_{*-r}(F)\ar[d]^{cl_F}\\
B_*(G)\ar[r]^{f_E^!}&B_{*-r}(F)}
\end{equation*}
\begin{proof}
 This follows from the previous proposition. Let us sketch the proof. We have to show that for any cycle $\alpha\in A_*(G)$ such that $cl_G\alpha=0$ we have that $cl_Ff^!\alpha=0$. Let $\alpha\in A_*(G)$ as above. By the definition of algebraic equivalence, there exists a non-singular variety $T$ of dimension $m$ and $k+m$-dimensional subvarieties $\mathcal{V}_i$ of  $T\times G$, flat over $T$ and points $t_1,\ t_2\in T$ such that
 \begin{equation*}
 \alpha=\sum_{i=1}^r[(V_i)_{t_1}]-[(V_i)_{t_2}].
 \end{equation*}
 By the above proposition we have that
 \begin{equation*}
 f_{t_j}^!(\sum_{i=1}^r[\mathcal{V}_i]_{t_j})=(f^!\sum_{i=1}^r[\mathcal{V}_i])_{t_j}
\end{equation*}
for $j=1,\ 2$. This shows that $$f^!\alpha=(f^!\sum_{i=1}^r[\mathcal{V}_i])_{t_1}-(f^!\sum_{i=1}^r[\mathcal{V}_i])_{t_2}.$$
Let us now analyze the right-hand side. Let $\mathcal{W}^{\prime}_i$ be cycles representing $f^![\mathcal{V}_i]$. As the pull back via $t_j$ is not influenced by components of $\mathcal{W}^{\prime}_i$ which do not map dominantly to $T$ we may discard them. Let us call the resulting stacks by $\mathcal{W}_i$. This shows that   $$f^!\alpha=(\sum_{i=1}^r[\mathcal{W}_i])_{t_1}-(\sum_{i=1}^r[\mathcal{W}_i])_{t_2}$$ and therefore $f^!\alpha$ is algebraically equivalent to zero.
\end{proof}

\begin{remark}
As $H_0(G)=B_0(G)$ the above morphism $f_E^!:B_*(G)\to B_{*-r}(F)$ induces a morphism which makes the diagram commute
\begin{equation*}
 \xymatrix{A_0(G)\ar[r]^{f_E^!}\ar[d]&A_{0-r}(F)\ar[d]\\
H_0(G)\ar[r]^{f_E^!}&H_{0-r}(F)}.
\end{equation*}
\end{remark}
\begin{remark} The definition of virtual pull-backs $i:X\to Y$ related to the cartesian diagram
\begin{equation*}
\xymatrix{X\ar[r]^i\ar[d]_q&Y\ar[d]\\
X^{\prime}\ar[r]&Y^{\prime}}
\end{equation*}
with $X^{\prime}\to Y^{\prime}$ a regular embedding and the obstruction bundle $E_{X/Y}:=q^*N_{X^{\prime}/Y^{\prime}}$ gives rise to a pull-back in homology $i_{E_{X/Y}}^!:H_*(Y)\to H_*(X)$ (see \cite{f}, Chapter 19).  We could not construct a similar morphism $$i_{E_{X/Y}}^!: H_*(Y)\to H_*(X)$$ in general.
\end{remark}

\end{proposition}
\section{Virtual push-forwards}\label{pfrd}
In this section we consider a proper surjective morphism $f:F\to G$ of stacks which possess perfect obstruction theories and we analyze the push-forward of the virtual class of $F$ along $f$. The strongest statement can only be obtained in \emph{algebraic equivalence} and therefore in this section we will work with homology groups instead of Chow groups. The main result of this section states that if the virtual dimension of $F$ is greater or equal to the virtual dimension of $G$ and the induced relative obstruction theory is perfect, then the push-forward of the virtual class of $F$ along $f$ is equal to a scalar multiple of the virtual class of $G$. This result is a generalization of the straight-forward fact that given a surjective morphism of schemes $f:F\to G$, with $G$ irreducible, then $f_*[F]$ is a scalar multiple of the fundamental class of $G$.
\\
\\Let us first formalize these ideas.
\begin{definition} Let $p : F\to G$ be a proper morphism of stacks possessing virtual classes $[F]^{\vv}\in A_{k_1}(F)$ and $[G]^{\vv}\in A_{k_2}(G)$ with $k_1\geq k_2$ and let $[G]^{\vv}_1,...,[G]^{\vv}_s\in A_{k_2}(G)$ be irreducible cycles such that $[G]^{\vv}=[G]^{\vv}_1+...+[G]^{\vv}_s$. Let $\gamma\in A^{k_3}(F)$, with $k_3\leq k_1-k_2$ be a cohomology class. We say that $p$ satisfies the virtual pushforward
property for $[F]^{\vv}$ and $[G]^{\vv}$ if the following two conditions hold:
\\(i) If the dimension of the cycle $\gamma \cdot[F]^{\vv}$ is bigger than the virtual dimension of
$G$ then $p_*(\gamma \cdot[F]^{\vv}) = 0$.
\\(ii) If the dimension of the cycle $\gamma\cdot [F]^{\vv}$ is equal to the virtual dimension of $G$
then $p_*(\gamma\cdot[F]^{\vv})=n_1[G]^{\vv}_1+...+n_s[G]^{\vv}_s$ for some $n_1,...,n_s\in\Q$.
\\We say the $p$ satisfies  the \emph{strong} virtual push-forward property if moreover, the following condition holds
\\($ii^{\prime}$) If the dimension of the cycle $\gamma\cdot [F]^{\vv}$ is equal to the virtual dimension of $G$
then $p_*(\gamma\cdot[F]^{\vv})$ is a scalar multiple of $[G]^{\vv}$.
\end{definition}
\begin{remark} The definition of ``push-forward property" first appears in Gathmann's work \cite{a}, with a minor difference. Gathmann says that a morphism satisfies the push-forward property if in our language it satisfies the strong virtual push-forward property. We prefer this terminology mainly because we would like to say that smooth morphisms satisfy the virtual push-forward property. 
\end{remark}
\begin{remark} Let $p : F\to G$ be a morphism as above. If $G$ is smooth of the expected dimension, then $p$ satisfies the virtual pushforward property. If $G$ is also irreducible, then $p$ satisfies  the strong virtual push-forward property.
\end{remark}

\begin{lemma}\label{vpf} Let $p:F\to G$ be a proper morphism of stacks possessing virtual classes of virtual dimensions $k_1$ respectively $k_2$ with $k_1\geq k_2$. If we have a compatible triple $(E^{\bullet}_{F/G},E^{\bullet}_{G},E^{\bullet}_{F})$, such that the relative obstruction theory $E^{\bullet}_{F/G}$ is perfect, then $p$ satisfies the virtual push-forward property.
\end{lemma}
\begin{proof}
The proof is very similar to Lai's arguments (see pages 9-11 in \cite{h}). 
\\Let $\mathfrak{E}_F:=h^1/h^0(E^{\bullet}_{F})$, $\mathfrak{E}_G:=h^1/h^0(E^{\bullet}_{G})$ and $\mathfrak{E}_{F/G}:=h^1/h^0(E^{\bullet}_{F/G})$. Let $0_{F}:F\to
\mathfrak{E}_F$, $0_{G}:F\to\mathfrak{E}_G$ and $0_{F/G}:F\to\mathfrak{E}_{F/G}$ be the zero-section embeddings. Then by the definition of the virtual class we have that $[G]^{\vv}=0_G^!C_{G}$ and $[F]^{\vv}=0_F^!C_{F}$. Let us denote by $G^{\mathrm{v}}$ any closed substack of $G$ such that $[G]^{\vv}=[G^{\mathrm{v}}]$ in $A_*(G)$. With this notation we have that 
\begin{equation}\label{notatie}
[G]^{\vv}=0_G^![\mathfrak{E}_{G}|_{G^{\mathrm{v}}}].
\end{equation} Let us now consider the following cartesian diagram
\begin{equation*}
 \xymatrix{F^{\prime}\ar[d]_i\ar[r]^q&G^{\mathrm{v}}\ar[d]\\F\ar[r]^p&G.}
\end{equation*}
By the proof of Theorem \ref{functoriality} we have that
\begin{align*}
[F]^{\vv}&=0_{F/G}^![C_{F^{\prime}/G^{\mathrm{v}}}]\\&=0_{F/G}^!0_{G}^![C_{F^{\prime}/G^{\mathrm{v}}}\times_{F^{\prime}}q^*\mathfrak{E}_{G}|_{G^{\mathrm{v}}}]\\&=0^!_{G}0_{F/G}^![C_{F^{\prime}/G^{\mathrm{v}}}\times_{F^{\prime}}q^*\mathfrak{E}_{G}|_{G^{\mathrm{v}}}]
\end{align*}
Let us denote $[C^{\prime}]:=0_{F/G}^![C_{F^{\prime}/G^{\mathrm{v}}}\times_{F^{\prime}}q^*\mathfrak{E}_{G}|_{G^{\mathrm{v}}}]\in A_*(q^*\mathfrak{E}_{G}|_{G^{\mathrm{v}}})$. Then the above computation shows that $$\gamma\cdot[F^{\vv}]=0^!_{G}\pi^*\gamma\cdot[C^{\prime}]$$ where $\pi:q^*\mathfrak{E}_{G}\to F$ denotes the canonical projection. By the commutativity of the pull-back with proper (projective) push-forward in the following cartesian diagram
\begin{equation*}
 \xymatrix{F^{\prime}\ar[r]\ar[d]&i^*p^*\mathfrak{E}_G\ar[d]\ar@/^1.5pc/[dd]^r\\F\ar[r]\ar[d]_p&p^*\mathfrak{E}_G\ar[d]\\G\ar[r]&\mathfrak{E}_G}
\end{equation*}
we obtain that
\begin{equation}\label{con}
p_*\gamma\cdot[F]^{\vv}=r_*\pi^*\gamma\cdot[C^{\prime}].
\end{equation}
By construction $C^{\prime}$ has a natural map to $\mathfrak{E}_{G}|_{G^{\mathrm{v}}}$ compatible with $r$ and therefore $p_*\gamma\cdot[F]^{\vv}=\sum n_i[\mathfrak{E}_{G}|_{G^{\mathrm{v}}}]_i$, where the sum is taken over all the irreducible components of $\mathfrak{E}_{G}|_{G^{\mathrm{v}}}$. We can now conclude the proof using equation \ref{notatie}.
\\If $k_3< k_1-k_2$, then $r_*\pi^*\gamma\cdot[C^{\prime}]=0$ for dimensional reasons and therefore $p_*\gamma\cdot[F]^{\vv}=0$.
\end{proof}
\begin{definition}\label{virtflat}
Let $p:F\to G$ be a surjective, proper morphism of stacks possessing virtual classes of virtual dimensions $k_1$ respectively $k_2$ with $k_1\geq k_2$. If we have a compatible triple $(E^{\bullet}_{F/G},E^{\bullet}_{G},E^{\bullet}_{F})$, such that the relative obstruction theory $E^{\bullet}_{F/G}$ is perfect, then we call $p$ a virtually smooth morphism.
\end{definition}
\begin{remark}
 This definition is very similar to Definition 3.14  in \cite{gf} of \emph{a family of proper virtually smooth schemes}. The main difference is that we do not ask the base $G$ to be smooth.
\end{remark}

\begin{theorem}\label{connected}
 Let $p:F\to G$ be a virtually smooth morphism. If $G$ is connected, then $p$ satisfies the strong virtual push-forward property.
\end{theorem}
\begin{proof}
In notations as above, we have by the above lemma that $$\gamma\cdot[F]^{\vv}=n_1[G^{\mathrm{v}}_1]+...+n_s[G^{\mathrm{v}}_s]$$ for some $n_1,...,n_s\in\Q$. Here $G^{\mathrm{v}}_1,...,G^{\mathrm{v}}_s$ are taken to be irreducible and such that $[G]^{\vv}=[G^{\mathrm{v}}_1]+...+[G^{\mathrm{v}}_s]$. We are left to show that all the $n_i$'s are equal.
\\Let $m_1,...,m_s$ be the geometric multiplicity of $G^{\mathrm{v}}_1,...,G^{\mathrm{v}}_s$. Then $[G]^{\vv}=m_1[G^{\mathrm{r}}_1]+...+m_s[G^{\mathrm{r}}_s]$, where $G^{\mathrm{r}}_i$ is the reduced stack associated to $G^{\mathrm{v}}_i$ and therefore $[C^{\prime}]=\sum_{i=1}^sm_i0_{F/G}^![C^{\prime}_i]$, where $C^{\prime}_i:=C_{F^{\prime}_i/G_i^{\mathrm{r}}}\times_{F^{\prime}_i}q^*\mathfrak{E}_{G}|_{G^{\mathrm{r}}_i}$. By equation (\ref{con}) we have that $p_*\gamma\cdot[F]^{\vv}=r_*\pi^*\gamma\cdot(\sum_{i=1}^sm_i[C_i^{\prime}])$. With this we have shown that it is enough to show the statement for $G$ reduced.
\\Let us consider the cartesian diagram
\begin{equation*}
 \xymatrix{X_P\ar[r]^{j}\ar[d]_{q_P}&F^{\prime}\ar[d]^q\\P\ar[r]^{i}&G^{\mathrm{v}}}
\end{equation*}
where $P$ is a general point in $G^{\mathrm{v}}$ and $X_P$ is the fiber of $q$ over $P$. As $G^{\mathrm{v}}$ is reduced we may assume that $P$ is a smooth point and therefore $i$ is a regular embedding. By the commutativity of pull-backs with proper push-forwards we have that
\begin{equation}
 (q_P)_*i^!\gamma\cdot[F]^{\vv}=i^*q_*\gamma\cdot[F]^{\vv}.
\end{equation}
This shows that $(q_P)_*i^!\gamma\cdot[F]^{\vv}=i^*\sum_in_i[G^{\mathrm{v}}_i]$. Without loss of generality we may assume that $P$ is a point on $G^{\mathrm{v}}_1$ and with this we obtain that
\begin{equation}\label{nr}
(q_P)_*i^!\gamma\cdot[F]^{\vv}=n_1[P].
\end{equation}
On the other hand by the commutativity of pull-backs we have that
\begin{align}\label{co}
 i^!q^![G]^{\vv}&=q_P^!i^*[G^{\mathrm{v}}]\\&=q_P^![P].
\end{align}
By the functoriality property of pull-backs we have that
\begin{equation}\label{fu}
 i^!q^![G^{\mathrm{v}}]=i^![F]^{\vv}.
\end{equation}
Equations (\ref{nr}), (\ref{co}) and (\ref{fu}) imply that
\begin{equation*}
 n_1[P]=(q_P)_*\gamma\cdot q_P^![P].
\end{equation*}
As $G$ is connected the right-hand side of the above equation does not depend on $P$, hence $p$ satisfies the push-forward property.
\end{proof}
\begin{remark}
 The only point where we need to work with homology is the last part of the proof of the above theorem. For any connected $G$ we have that $H_0(G)=\Q$, but this is usually no longer true for the corresponding Chow group.
\end{remark}

\begin{remark}\label{induced}
 Let us consider a cartesian diagram of stacks
\begin{equation*}
 \xymatrix{F^{\prime}\ar[r]\ar[d]_q&F\ar[d]^p\\
G^{\prime}\ar[r]^i&G.}
\end{equation*}
If $F$, $G$, $G^{\prime}$ posses virtual classes and the relative obstruction theory $E^{\bullet}_{F/G}$ is perfect, then we have an induced virtual class on $F^{\prime}$, namely $[F^{\prime}]^{\vv}:=q^![G^{\prime}]^{\vv}$.
\end{remark}

\begin{corollary}\label{basech}
 Let us consider a cartesian diagram of stacks
\begin{equation*}
 \xymatrix{F^{\prime}\ar[r]\ar[d]_q&F\ar[d]^p\\
G^{\prime}\ar[r]^i&G}
\end{equation*}
such that $p$ is proper and $F$, $G$, $G^{\prime}$ posses virtual classes. If the relative obstruction theory $E^{\bullet}_{F/G}$ is perfect, $G$ is connected, then $q$ satisfies the strong virtual push-forward property for $[F^{\prime}]^{\vv}$ and $[G^{\prime}]^{\vv}$, where $[F^{\prime}]^{\vv}$ is the one defined in Remark \ref{induced}.
\end{corollary}
\begin{proof}
By Lemma \ref{vpf} we have that $p_*(\gamma\cdot[F^{\prime}]^{\vv})=n_1[G^{\prime}]^{\vv}_1+...+n_s[G^{\prime}]^{\vv}_s$ for some $n_1,...,n_s\in\Q$. We have to show that all $n_i$'s are equal.
\\As is the proof of the theorem we may assume that $G$ and $G^{\prime}$ are reduced.
\\Let us consider the following cartesian diagram
 \begin{equation*}
 \xymatrix{X\ar[r]\ar[d]_{q_P}&F^{\prime}\ar[r]\ar[d]_q&F\ar[d]^p\\
P\ar[r]&G^{\prime}\ar[r]^i&G}
\end{equation*}
where $P$ is any closed point. Let $\gamma\in A^*(F)$ as in \ref{connected}. Then, by Theorem \ref{connected}, we have that $$p_*\gamma\cdot[F]^{\vv}=n[G]^{\vv}$$ for some $n\in\Q$. Also, by the proof of \ref{connected}, we have that $$(q_P)_*\gamma\cdot [X]^{\vv}=n[P].$$ Looking now at the diagram on the left, and assuming that $P$ is a smooth point of $G^{\prime}$ we obtain the following by Theorem \ref{relations}
\begin{equation*}
 q_*(\gamma\cdot[F^{\prime}]^{\vv})=(q_P)_*\gamma\cdot [X]^{\vv}.
\end{equation*}
As $G^{\prime}$ is reduced, we have that the generic point is smooth and hence the above equation holds for a dense open subset of $G^{\prime}$. Combining this equation with the previous, we obtain that $q_*(\gamma\cdot[F^{\prime}]^{\vv})=n[G^{\prime}]^{\vv}$.
\end{proof}
\section{Applications}
\subsection[Conservation of number]{Conservation of number for virtually smooth morphisms}
Let us recall Fulton's principle of conservation of number (see Proposition 10.2 \cite{f}).
\begin{proposition}Let $f:F\to G$ be a proper morphism, $G$ an $m$-dimensional irreducible scheme. Let $i_P:P\to G$ be a point in $G$ and $\alpha$ be an $m$-dimensional cycle on $F$. Then the cycle classes $\alpha_P:=i_P^*\alpha$ have the same degree.
\end{proposition}
In this section we will give a version of this principle in the situation when $f:F\to G$ is a virtually smooth morphisms.
\\As a consequence of the conservation of number principle we give a proof of the fact that the virtual Euler characteristic is constant in virtually smooth families (see Definition \ref{virtflat}). This statement is a generalization of Proposition 4.14 in \cite{gf} of Fantechi and G\"ottsche.
\\As in the section on virtual push-forwards we work with \emph{homology} rather than with Chow groups.
\\
\\Let us now state the conservation of number principle for virtually smooth morphisms.
\begin{proposition}\label{ctnr}
 Let $G$ be a connected stack of pure dimension and let $f:F\to G$ be a proper virtually smooth morphism of stacks (see Definition \ref{virtflat}) of virtual relative dimension $d$. Let $i:P\to X$ be a point in $X$ and let us consider $\alpha\in A^d(F)$. Then, the number $$i^*\alpha\cdot [X_P]^{\vv}$$ is constant. 
\end{proposition}
\begin{proof}
Let $P$ be any point of $G$ and let us consider the following cartesian diagram
\begin{equation*}
 \xymatrix{X_P\ar[d]_g\ar[r]^j&F\ar[d]^f\\P\ar[r]^i&G}
\end{equation*}
where $X_P$ is the fiber of $X$ over $P$ and $g:X_P\to P$ is the map induced by $f$. By Theorem \ref{connected} we have that
\begin{equation}
\alpha[F]^{\vv}=n[G].
\end{equation}
Let us show that $n$ is equal to $i^*\alpha\cdot[X_P]^{\vv}$ for any $P$. For this, we see that
\begin{align*}
 i_*g_*i^*\alpha\cdot[X_P]^{\vv}&=f_*(j_*(j^*\alpha)\cdot[X_P]^{\vv})\\&=f_*(\alpha\cdot j_*[X_P]^{\vv}).
\end{align*}
As the $G$ is connected it follows that $j_*[X_P]^{\vv}$ does not depend on the point $P$ and therefore the intersection product $i^*\alpha\cdot[X_P]^{\vv}$ is equal to $n$ for any $P$.
\end{proof}

\begin{remark}
Taking $G$ to be smooth we obtain the conservation of number principle in families of virtually smooth schemes (see definition 3.14 in \cite{gf})  which is Corollary 3.16 in \cite{gf}.

\end{remark}

\subsection{Virtual Euler characteristics in virtually smooth families}

\begin{definition}Let $f:F\to G$ be a morphism of proper stacks with a 1-perfect obstruction theory $E_{F/G}$ which admits a global resolution of $E_{F/G}$ as a complex of vector bundles $[E^{−1}\to E^0]$ (e.g. if $F$ can be embedded as closed substack in a separated stack which is smooth over $G$.)
We denote by $[E_0\to E_1]$ the dual complex and by $d$ the expected dimension $d :=
rkE_{F/G} = rk E^0-rk E^{−1}$. We denote the class $[E_0]-[E_1]\in K^0(F)$ by $T^{\vv}_{F/G}$ and we call it the virtual relative tangent of $f$.
\end{definition}
\begin{definition}Let $f:F\to G$ be a morphism of stacks as before.
 We define the relative virtual Euler characteristic of $f$ to be the top
virtual Chern number $e^{\vv}(F/G) := c_d(T^{\vv}_{F/G})$.
\end{definition}
\begin{remark}
 The definition is coherent with Definition 4.2 in \cite{gf} by Corollary 4.8 ((Hopf index theorem) in \cite{gf}.
\end{remark}
\begin{proposition}\label{constant}
 Let $G$ be a connected stack of pure dimension and let $f:F\to G$ be a morphism of stacks with $E_{F/G}$ a perfect obstruction theory for $f$. Then, all the fibers of $f$ have the same virtual Euler characteristic. 
\end{proposition}
\begin{proof}
 We use the above proposition with $\alpha:=c_d(T^{\vv}_{F/G})$.
\end{proof}
\begin{remark}
Taking $G$ to be smooth we obtain that the virtual Euler characteristic is constant in a family of virtually smooth schemes. This is a different proof of Proposition 4.14 in \cite{gf}.
\end{remark}

\subsection{Virtual push-forward and Gromov-Witten invariants}\label{gw}
\paragraph{The standard obstruction theory for the moduli space of stable maps.} Let us fix notations. Let $X$ be a smooth projective variety and $\beta\in A_1(X)$ a homology class of a curve in $X$. We denote by $\overline{M}_{g,n}(X,\beta)$ the moduli space of stable genus-$g$, $n$-pointed maps to $X$ of homology class $\beta$. Let $$\epsilon_X: \overline{M}_{g,n}(X,\beta)\to \mathfrak{M}_{g,n}$$ be the morphism that forgets the map (and does not stabilize the pointed curve) and $$\pi_X: \overline{M}_{g,n+1}(X,\beta)\to \overline{M}_{g,n}(X,\beta)$$ the morphism that forgets the last marked point and stabilizes the result. Then it is a well-known fact that  $$E_{\overline{M}_{g,n}(X,\beta)/\mathfrak{M}}^{\bullet}:=(\mathcal{R}^{\bullet}(\pi_{X})_*ev_{X}^*T_{X})^{\vee}$$ defines an obstruction theory for the morphism $p$, where $ev_X$ indicates the evaluation map $ev_X:\overline{M}_{g,n+1}(X,\beta)\to X$ (see \cite{b}). We call $$[\overline{M}_{g,n}(X,\beta)]^{\vv}:=(\epsilon_X)_{\mathfrak{E}_{\overline{M}_{g,n}(X,\beta)/\mathfrak{M}}}^!\mathfrak{M}_{g,n}$$ the virtual class of $\overline{M}_{g,n}(X,\beta)$.
\begin{remark}\label{welldef}
Let $p: X\to Y$ be a morphism of smooth algebraic varieties. Let $\beta\in H_2(X)$ and $g,\ n$ be any natural numbers such that
\begin{itemize} 
\item either $g\geq 2$
 \item either $g< 2$ and $f_*\beta\neq0$
\item either $g=1$, $f_*\beta=0$ and $n\geq 1$, either $g=0$, $p_*\beta=0$ and $n\geq 3$.
\end{itemize}
Then $p$ induces a morphism of stacks 
\begin{align*}\bar{p}: \overline{M}_{g,n}(X,\beta)&\to \overline{M}_{g,n}(Y,p_*\beta)\\
(\tilde{C},x1,...,x_n,\tilde{f}&\mapsto(C,x_1,...,x_n,p\circ\tilde{f}))
\end{align*}
where $C$ is obtain by $\tilde{C}$ by contracting the unstable components of $f:=p\circ\tilde{f}$
\\Convention: Given a morphism of smooth algebraic varieties $f:X\to Y$, we will indicate the induced morphism between moduli spaces of stable maps by the same letter with a bar.
\\Convention: For simplicity, we will denote obstruction theories of a morphism $f:F\to G$ by $E_f$. For example, we will write $E_{\epsilon_X}$ instead of $E_{\overline{M}_{g,n}(X,\beta)/\mathfrak{M}}^{\bullet}$.
\\Convention: In the following, everytime we write $\bar{p}:\overline{M}_{g,n}(X,\beta)\to \overline{M}_{g,n}(Y,p_*\beta)$ we will assume that $\overline{M}_{g,n}(Y,p_*\beta)$ is non-empty.
\end{remark}
\begin{remark} The moduli space $\overline{M}_{g,n}(X,\beta)$ has a perfect dual \emph{absolute} obstruction theory
$$0\to\mathcal{T}^1\to(E_0)_X\to (E_1)_X\to \mathcal{T}^2\to0.$$ Let us fix a point $P:=(C,x_1,...,x_n)$ and let us denote by $T_P$ the restriction of $\mathcal{T}^1$ to $P$ and by $Ob_P$ the restriction of $\mathcal{T}^2$ to $P$. Then we have the following exact sequence
\begin{align*}
0\to Ext^0 (\Omega_C (D), \mathcal{O}_C )\to H^0 (C, f^*T_ X ) \to T_P\to \\
\to Ext^1 (\Omega_C (D), \mathcal{O}_C ) \to H^ 1 (C, f^*T _X ) \to Ob_P\to 0. \end{align*}
\end{remark}
\paragraph{Costello's construction.} For our purposes it is easier to use Costello's trick (\cite{c}). Let us shortly present how his construction applies to our case. In section 2 of \cite{c}, Costello introduces an artin stack $\mathfrak{M}_{g,n, \beta}$, where $\beta$ is an additional labeling of each irreducible components of a marked curve of genus $g$ by the elements of a semigroup. We will take this semigroup to be $H_2(X)$, for some smooth variety $X$. In \cite{c} it is shown that the forgetful map $\mathfrak{M}_{g,n, \beta}\to \mathfrak{M}_{g,n}$ is \'etale and that the natural forgetful map $\epsilon_X:\overline{M}_{g,n}(X,\beta)\to \mathfrak{M}_{g,n}$ factors through $\epsilon_{X,\beta}:\overline{M}_{g,n}(X,\beta)\to\mathfrak{M}_{g,n, \beta}$. Therefore, we can consider the perfect relative obstruction theory of $\overline{M}_{g,n}(X,\beta)$ $$\mathcal{R}^{\bullet}\pi_*f^*T_X\to \mathfrak{M}_{g,n, \beta}$$
which induces a virtual class $([\overline{M}_{g,n}(X,\beta)]^{\vv})^{\prime}:=\epsilon_{\beta}^![\overline{M}_{g,n}(X,\beta)]$. It can be easily seen that $$([\overline{M}_{g,n}(X,\beta)]^{\vv})^{\prime}=[\overline{M}_{g,n}(X,\beta)]^{\vv}.$$
This construction has the advantage that for a given map $p:X\to Y$ we have a commutative diagram
\begin{equation}
\xymatrix{\overline{M}_{g,n}(X,\beta)\ar[d]_{\bar{p}}\ar[r]^{\epsilon_{X,\beta}}&\mathfrak{M}_{g,n, \beta}\ar[d]^{\psi}\\
\overline{M}_{g,n}(Y,p_*\beta)\ar[r]^{\epsilon_{Y,p_*\beta}}&\mathfrak{M}_{g,n, p_*\beta}}
\end{equation}
where $\psi(C)$ contracts the unstable components $C_i$ such that the label $\beta_i$ satisfies $p_*\beta_i=0$ and changes the label on each irreducible component by $p_*\beta_i$.
%\begin{lemma}
%In notations as in Remark \ref{welldef} there are canonical isomorphisms $$H^0(C,f^*T_Y)\simeq H^0(\tilde{C},\tilde{f}^*p^*T_Y),$$ $$H^1(C,f^*T_Y)\simeq H^1(\tilde{C},\tilde{f}^*p^*T_Y).$$
%\end{lemma}
\begin{proposition}
If $p:X\to Y$ is a smooth morphism, then the relative obstruction theory of $\bar{p}: \overline{M}_{g,n}(X,\beta)\to\overline{M}_{g,n}(Y,p_*\beta)$ is perfect.
\end{proposition}
\begin{proof} By the discussion in the above paragraph we have that $\mathfrak{M}_{g,n, p_*\beta}$ is \'etale over $\mathfrak{M}_{g,n}$ and therefore $$[\overline{M}_{g,n}(X,\beta)]^{\vv}=\epsilon_{X,\beta}^![\mathfrak{M}_{g,n, \beta}]$$ and similarly $$[\overline{M}_{g,n}(Y,p_*\beta)]^{\vv}=\epsilon_{Y,p_*\beta}^![\mathfrak{M}_{g,n, \beta}].$$
\\\textit{Step 1.} Let us consider the following exact sequence
\begin{equation*}
p^*\Omega_Y\to \Omega_X\to \Omega_{X/Y}
\end{equation*}
and let us look at the induced distinguished triangle
\begin{equation}\label{derivedseq}
\mathcal{R}^{\bullet}\pi_*ev^*_Xp^*\Omega_Y\to \mathcal{R}^{\bullet}\pi_*ev^*_X\Omega_X\to\mathcal{R}^{\bullet}\pi_*ev^*_X\Omega_{X/Y}\to\mathcal{R}^{\bullet}\pi_*ev^*_Xp^*\Omega_Y[1].
\end{equation}
% On the other hand we have a distinguished triangle
%\begin{equation}\label{obsseq}
%\begin{array}{ll}
%\bar{p}^*E^{\bullet}_{\overline{M}_{g,n}(Y,p_*\beta)/\mathfrak{M}_{g,n, p_*\beta}} & \to E^{\bullet}_{\overline{M}_{g,n}(X,\beta)/\mathfrak{M}_{g,n, p_*\beta}}\to E^{\bullet}_{\overline{M}_{g,n}(X,\beta)/\overline{M}_{g,n}(Y,p_*\beta)}\to \\ & \\
%\bar{p}^*E^{\bullet}_{\overline{M}_{g,n}(Y,p_*\beta)/\mathfrak{M}_{g,n, p_*\beta}}[1] & .
%\end{array}
%\end{equation}
By Corollary 5.3 in \cite{eu}, we have that $\mathcal{R}^{\bullet}\pi_*ev^*_Xp^*\Omega_Y\simeq p^*\mathcal{R}^{\bullet}\pi_*ev^*_Y\Omega_Y$. In notations as in the beginning of the section we can rewrite triangle \ref{derivedseq} as
\begin{equation}
p^*E_{\epsilon_{Y,p_*\beta}}\to E_{\epsilon_{X,\beta}}\to\mathcal{R}^{\bullet}\pi_*ev^*_X\Omega_{X/Y}\to p^*E_{\epsilon_{Y,p_*\beta}}[1].
\end{equation}
Let us note that all complexes are perfect.
%Using the same argument we have that  and that $$[\overline{M}_{g,n}(X,\beta)]^{\vv}=(\psi\circ\epsilon_{X,\beta})^![\mathfrak{M}_{g,n, p_*\beta}].$$ 
%\\By the commutativity of the above diagram we have $$(\psi\circ\epsilon_{X,\beta})^![\mathfrak{M}_{g,n, p_*\beta}]=(\epsilon_{Y,\beta}\circ\bar{p})^![\mathfrak{M}_{g,n, p_*\beta}].$$
\\\textit{Step 2.} The morphism obtained from the following composition $$E_{\epsilon_{X,\beta}}[-1]\to L_{\epsilon_{X,\beta}}[-1]\to\epsilon_{X,\beta}^* L_{\psi}$$  can be completed to a triangle $$E_{\epsilon_{X,\beta}}[-1]\to \epsilon_{X,\beta}^*L_{\psi}\to E_{\psi\circ\epsilon_{X,\beta}}\to E_{\epsilon_{X,\beta}}.$$ By the axioms of triangulated categories we obtain a morphism $E:=E_{\psi\circ\epsilon_{X,\beta}}\to L_{\psi\circ\epsilon_{X,\beta}}$, which it can be easily seen to be an obstruction theory to the morphism $\psi\circ\epsilon_{X,\beta}$. In a similar way we obtain  a complex which we denote by $E_{\bar{p}}$ such that the triangle
\begin{equation}
\bar{p}^*E_{Y,p_*\beta}\to E\to E_{\bar{p}}\to \bar{p}^*E_{Y,p_*\beta}[1]
\end{equation}
is distinguished. By the octahedron axiom we obtain that the triangle
\begin{equation}
\epsilon_{X,\beta}^*L_{\psi}\to E_{\bar{p}}\to \mathcal{R}^{\bullet}\pi_*ev^*_X\Omega_{X/Y}\to \epsilon_{X,\beta}^*L_{\psi}[1]
\end{equation}
is distinguished. From the long exact sequence in cohomology and the fact that $h^{-2}(\mathcal{R}^{\bullet}\pi_*ev^*_X\Omega_{X/Y})=0$ we obtain that $E_{\bar{p}}$ is a perfect obstruction theory for the morphism $\bar{p}$. This shows the claim.
\end{proof}
\begin{proposition}\label{projbdl} Let $p:X\to\pp^r$ be a smooth morphism.  If $\bar{p}$ has strictly positive virtual relative dimension, then $\bar{p}: \overline{M}_{g,n}(X,\beta)\to\overline{M}_{g,n}(\pp^r,p_*\beta)$ satisfies the strong push-forward property.
\end{proposition}
\begin{proof} By \cite{kp} $\overline{M}_{g,n}(\pp^r,p_*\beta)$ is connected and by Theorem \ref{connected} $\bar{p}$
 satisfies the strong virtual push-forward property.
 \end{proof}
 \begin{proposition}Let $L_1,...,L_s$ be very ample line bundles on a smooth projective variety $X$ and let us consider a projective bundle $p:\pp_X(\oplus L_i)\to X$. Then the induced morphism $\bar{p}: \overline{M}_{g,n}(\pp_X(\oplus L_i),\beta)\to\overline{M}_{g,n}(X,p_*\beta)$ satisfies the strong push-forward property.
 \end{proposition}
\begin{proof}Let us consider $j_i:X\to\pp^{r_i}$ to be the embedding of $X$ into a projective space induced by the line bundle $L_i$. Then we have a Cartesian diagram
\begin{equation*}
\xymatrix{\pp_X(\oplus L_i)\ar[r]\ar[d]&\pp_{\pp^{r_1}\times...\times\pp^{r_s}}(\oplus\mathcal{O}(1))\ar[d]\\
X\ar[r]^{j_1\times...\times j_s\ \ \ \ \ \ \ \ }&\pp^{r_1}\times...\times\pp^{r_s}}
\end{equation*}
The conclusion follows by the above proposition and Corollary \ref{basech}.
\end{proof}
\subsection{Stable maps and stable quotients}
In this section we want to analyze the push forward of the virtual class of the moduli space of stable maps $\bar{M}_{g,n}(\G(1,n),d)$ along the morphism $$c:\bar{M}_{g,m}(\G(1,n),d)\to \bar{Q}_{g,m}(\G(1,n),d)$$ which was introduced in \cite{mop}. Let us briefly recall the basic definitions.

\paragraph{Stable quotients.} Let $(C,p_1,...,p_m)$ be a nodal curve of genus $g$ with $m$ distinct markings which are different from the nodes. A quotient on $C$ $$0\to S\to\mathcal{O}_C^n\stackrel{q}{\rightarrow} Q$$ is called \textit{quasi-stable} if the torsion sheaf $\tau(Q)$ is not supported on nodes or markings. Let $r$ be the rank of $S$. A quotient $(C,p_1,...,p_m,q)$ is called stable if $$\omega_C(p_1+...+p_m)\otimes (\wedge^rS^{\vee})^{\epsilon}$$ is ample on $C$ for every strictly positive $\epsilon \in\Q$.
\begin{remark}The space of stable quotients $\bar{Q}_{g,m}(\G(r,n),d)$ is an other compactification of the space of genus $g$ curves (with $m$ marks) in the Grassmannian $\G(r,n)$. This can be easily seen from the universal property of the tautological sequence on the Grassmannian: to give a curve $C\stackrel{i}\hookrightarrow\G(r,n)$ is equivalent to giving a quotient $$0\to i^*S\to\mathcal{O}_C^n\to i^*Q,$$ where
\begin{equation*}
0\to S\to \mathcal{O}^n\to Q\to 0
\end{equation*}
is the tautological sequence on the Grassmannian.
\end{remark}
\paragraph{Obstruction theory.} As the moduli space of stable maps, the moduli space of stable quotients $\bar{Q}_{g,m}(\G(r,n),d)$ has a morphism $\nu:\bar{Q}_{g,m}(\G(r,n),d)\to\mathfrak{M}_{g,m}$to the Artin stack of nodal curves. Let $p:\bar{U}\to\bar{Q}_{g,m}(\G(r,n),d)$ be the universal curve over $\bar{Q}_{g,m}(\G(r,n),d)$ and let $$0\to\mathcal{S}\to\mathcal{O}_{\bar{U}}^n\to\mathcal{Q}\to0$$ be the universal sequence on $\bar{U}$. Then the complex $$Rp_*RHom(\mathcal{S},\mathcal{Q})$$ is the obstruction theory relative to $\nu$.
\paragraph{Stable maps and stable quotients.}\label{mapstoquot}
\begin{proposition}
When $r=1$ there exists a map $c:\bar{M}_{g,m}(\G(1,n),d)\to \bar{Q}_{g,m}(\G(1,n),d)$ extending the isomorphism on smooth curves.
\end{proposition}
\begin{proof} This has been proved in \cite{mop} and a similar situation appears in \cite{mihnea}. Let us shortly sketch the proof. Let $(\pi_Y,f):Y\to S\times\G(1,n)$ be a family of stable maps to $\G(1,n)$. As described in the above remark, this comes with an exact sequence $$0\to f^*S\to\mathcal{O}^n\to f^*Q\to0.$$
Let $\pi_X: X\to S$, be the family of curves obtained by contracting all rational trees with no marked points and let $q:Y\to X$ be the contracting morphism. In the following, we will give a canonical way to associate a quasi-stable quotient to the family $\pi_X$.
We denote by
\begin{equation*}
0\to \mathcal{S}\to \mathcal{O}^n\to \mathcal{Q}\to 0
\end{equation*}
the tautological sequence on the universal curve over $\bar{Q}_{g,m}(\G(1,n),d)$.
Let $E$ be either a divisor or a component of $\bar{M}_{g,m+1}(\G(1,n),d)$ such that the general element of the map to $\bar{M}_{g,m}(\G(1,n),d)$ is an irreducible rational curve, with the additional condition that this general fiber touches only one other curve in the domain of the map it is associated to. For each such locus $E$ there is a well defined line bundle $\mathcal{O}(E)$ on $\bar{M}_{g,m+1}(\G(1,n),d)$. This line bundle has degree $-1$ when restricted to the general fiber of the induced map from $E$ to 
$\bar{M}_{g,m}(\G(1,n),d)$. We attach the weight $\delta$ to such a $E$ if the degree of $\mathcal{S}$ restricted to the 
general fiber in $D$ is $-\delta$. We consider the bundle $$\mathcal{S}^{\prime}:=S\otimes\mathcal{O}(-\delta E)$$ which is trivial along the rational tails. Then it can be showed that $q_*\mathcal{S}^{\prime}$ is a stable quotient. 
\end{proof}
\begin{remark}The above map associates to a map $f:C\to\G(1,n)$, the curve $\widehat{C}$ obtained by contracting the rational tails and the exact sequence $$0\to S(-\sum d_ix_i)\to \mathcal{O}_{\widehat{C}}^n\to\widehat{Q}\to0,$$
where $x_i$ are the points on $C$ where the rational trees glue the rest of the curve and $d_i$ is the degree of $f$ on the tree $C_i$.
\end{remark}
%\begin{lemma}\label{pullback}There exists an isomorphism $d^*\mathcal{S}\simeq f^*S\otimes \mathcal{O}(-E)$, where by $E$ we denote the divisor on $\bar{M}_{g,m}(\G(1,n),d)$ obtained in the same fashion as $D$.
%\end{lemma}
\paragraph{Compatibility of obstructions.}
\begin{lemma}There exists a morphism $$R\pi_*f^*T_{\G(1,n)}\to c^*Rp_*RHom(\mathcal{S},\mathcal{Q}).$$
\end{lemma}
\begin{proof}
The complex $R\pi_*f^*T_{\G(1,n)}$ on $\bar{M}_{g,m}(\G(1,n),d)$ is given by the following:
\\ (i) For every family of maps $\pi_X:X\to S$ the complex $R\pi_*g_X^*Q\otimes S^{\vee}$, where $g$ is the map induced by the morphism $S\to\bar{M}_{g,m}(\G(1,n),d)$ and the evaluation $ev_{m+1}\bar{M}_{g,m+1}(\G(1,n),d)\to\G(1,n)$.
\\ (ii)For every morphism $\varphi:S^{\prime}\to S$ the canonical isomorphism $$R(\pi_{X^{\prime}})_*g_{X^{\prime}}^*Q\otimes S^{\vee}\simeq \varphi^*R(\pi_X)_*g_X^*Q\otimes S^{\vee}.$$
Similarly, the complex $Rp_*RHom(\mathcal{S},\mathcal{Q})$ is constructed by giving for any family $\pi_Y:Y\to S$ a complex $R(p_Y)_*RHom(\mathcal{S}_Y,\mathcal{Q}_Y)$ and the obvious isomorphisms.
\\In the following we will construct the desired morphism on families.
\\Let us first write $Rp_*RHom(\mathcal{S},\mathcal{Q})$ in a different form. Using the basic compatibilities of derived functors and the fact that $\mathcal{S}$ is a line bundle we obtain
\begin{align*}
Rp_*RHom(\mathcal{S},\mathcal{Q})&=Rp_*RHom(\mathcal{O}_{\bar{U}},\mathcal{S}^{\vee}\otimes\mathcal{Q})\\
&=Rp_*RHom(p^*\mathcal{O}_{\bar{Q}},\mathcal{S}^{\vee}\otimes\mathcal{Q})\\
&=RHom(\mathcal{O}_{\bar{Q}},Rp_*\mathcal{S}^{\vee}\otimes\mathcal{Q})\\
&=Rp_*(\mathcal{S}^{\vee}\otimes\mathcal{Q}).
\end{align*}
Consider the following diagram
\begin{equation*}
\xymatrix{Y\ar[r]\ar[d]^q\ar@/ ^-1.5pc/[dd]_{\pi_Y}&\bar{M}_{g,m+1}(\G(1,n),d)\ar[dd]\\
X\ar[rr]\ar[d]^{\pi_X}&&\bar{U}\ar[d]\\
S\ar[r]&\bar{M}_{g,m}(\G(1,n),d)\ar[r]^c&\bar{Q}_{g,m}(\G(1,n),d)}
\end{equation*}
where $q$ is the morphism contracting rational tails induced by $c$.
From the above we see that we need to construct a morphism from $$R(\pi_Y)_*g^*Q\otimes S^{\vee}\to R(\pi_X)_*\mathcal{Q}\otimes\mathcal{S}^{\vee}.$$
%Let us now look at the behavior of $\varphi$ on the fibers. It can be easily seen that the fiber of $c^*Rp_*RHom(\mathcal{S},\mathcal{Q})$ in a point $(C, x_1,...,x_n, f)$ is $H^i(C, (f^*S(-\sum d_ix_i)^{\vee}\otimes Q)$. On the other hand the above computation and lemma \ref{pullback} show that the fiber of $R\pi_*d^*RHom(\mathcal{S},\mathcal{Q})$ is again $H^i(C, (f^*S(-\sum d_ix_i)^{\vee}\otimes Q)$. By the definition of $\varphi$ we see that it is the identity on fibers. This shows that we have an isomorphism $$c^*Rp_*RHom(\mathcal{S},\mathcal{Q})\simeq R\pi_*d^*RHom(\mathcal{S},\mathcal{Q}).$$
By the construction of the morphism $c$ we have that
\begin{equation*}
\mathcal{S}=q_*S(-E).
\end{equation*}
This shows that we have a morphism 
\begin{equation}\label{multiplic}q_*S^{\vee}\to\mathcal{S}^{\vee}.
\end{equation} Consider the following exact sequences
\begin{equation*}
0\to g^*\mathcal{O}\to g^*(S^{\vee})^{\oplus n}\to g^*T_{\G}\to0
\end{equation*}
on $Y$ and
\begin{equation*}
0\to\mathcal{O}\stackrel{j}{\rightarrow}(\mathcal{S}^{\vee})^{\oplus n}\to\mathcal{S}^{\vee}\otimes\mathcal{Q}\to0
\end{equation*}
on $X$. Pushing forward the first exact sequence and using (\ref{multiplic}) and that fact that $Rq_*\mathcal{O}_Y\simeq[\mathcal{O}_X]$ we obtain a morphism of distinguished triangles
\begin{equation*}
\xymatrix{Rq_*[\mathcal{O}_Y]\ar[r]\ar[d]&Rq_*[(g^*S^{\vee})^{\oplus n}]\ar[r]\ar[d]&Rq_*[g^*T_{\G}]\ar[d]\\
[\mathcal{O}_X]\ar[r]&[(g^*\mathcal{S}^{\vee})^{\oplus n}]\ar[r]& [\mathcal{S}^{\vee}\otimes\mathcal{Q}]}.
\end{equation*}
Here the brackets $[A]$ indicate a complex with the sheaf $A$ concentrated in zero. Pushing forward the last column along $\pi_X$ we obtain the required morphism $R(\pi_Y)_*g^*Q\otimes S^{\vee}\to R(\pi_X)_*\mathcal{Q}\otimes\mathcal{S}^{\vee}$. It can be checked that it is compatible with restrictions and it gives a morphism of complexes $R\pi_*f^*T_{\G(1,n)}\to c^*Rp_*RHom(\mathcal{S},\mathcal{Q})$ in the derived category of the moduli stack $\bar{M}_{g,m}(\G(1,n),d)$.

%a morphism $Rq_*g^*Q\otimes S^{\vee}\to \mathcal{Q}\otimes\mathcal{S}^{\vee}$

%Pulling back the last one via $d$ we obtain an exact sequence on $\bar{M}_{g,m+1}(\G(1,n),d)$
%\begin{equation*}
%0\to f^*\mathcal{O}\to(f^*S^{\vee}(E))^{\oplus n}\to d^*(\mathcal{S}^{\vee}\otimes\mathcal{Q})\to0.
%\end{equation*}
%This shows that we have  a morphism of distinguished triangles
\end{proof}
\begin{lemma} Let $F$ be the cone of the morphism $$R\pi_*f^*T_{\G(1,n)}\to c^*Rp_*RHom(\mathcal{S},\mathcal{Q}).$$ Then, $F$ is a perfect complex.
\end{lemma}

\begin{proof}
Let us consider $$f:(C,x_1,...,x_{m})\to\G$$ a stable map, $q:C\to\widehat{C}$ the morphism contracting the rational tails and let $x_1,...,x_p$ be the gluing points of the rational tails with the rest of the curve. On $\widehat{C}$ we have an induced stable quotient $$0\to S(-\sum d_ix_i)\to\mathcal{O}_{\widehat{C}}^{\oplus n}\to \widehat{Q}\to0.$$  Then we need to show that the morphism
\begin{equation*}
H^1(C, f^*S^{\vee}\otimes Q)\to H^1(\widehat{C}, S^{\vee}(\sum d_ix_i)\otimes \widehat{Q})
\end{equation*}
is surjective. Since $$H^1(\widehat{C}, f^*S^{\vee}(\sum d_ix_i)\otimes \widehat{Q}\simeq H^1(C, q^*S^{\vee}(\sum d_ix_i)\otimes \widehat{Q})$$ we need to show that 
\begin{equation*}
H^1(C, f^*S^{\vee}\otimes Q)\to H^1(C, f^*S^{\vee}(\sum d_ix_i)\otimes \widehat{Q})
\end{equation*}
is surjective. As the quotient of the morphism $S^{\vee}\otimes Q\to S^{\vee}(\sum d_ix_i)\otimes \widehat{Q}$ is supported in dimension zero, it has no higher cohomology. This shows that the above morphism is surjective.
\end{proof}

\begin{proposition}\label{mapquot}We have that
$$c_*[\bar{M}_{g,n}(\G(1,n),d)]^{\vv}=\bar{Q}_{g,n}(\G(1,n),d).$$
\end{proposition}
\begin{proof} Let us consider the following commutative diagram
\begin{equation}\label{moregeneral}
\xymatrix{\bar{M}_{g,n}(\G(1,n),d)\ar[r]^c\ar[d]^{\epsilon}&\bar{Q}_{g,n}(\G(1,n),d)\ar[d]^{\nu}\\
\mathfrak{M}\ar[r]^{\mu}&\mathfrak{M}}
\end{equation}
where $\mu$ is the map contracting the rational tails. The morphism obtained from the following composition $$E_{\epsilon}[-1]\to L_{\epsilon}[-1]\to\epsilon^* L_{\mu}$$  can be completed to a triangle $$E_{\epsilon}[-1]\to \epsilon^*L_{\mu}\to E_{\mu\circ\epsilon}\to E_{\epsilon},$$ where $E_{\epsilon}$ indicates the obstruction to the morphism $\epsilon$ described in section \ref{gw}. By the axioms of triangulated categories we obtain a morphism $E:=E_{\mu\circ\epsilon}\to L_{\mu\circ\epsilon}$, which it can be easily seen to be an obstruction theory to the morphism $\mu\circ\epsilon$. In a similar way we obtain a complex which we denote by $E_c$ such that the triangle
\begin{equation}\label{thetrueone}
c^*E_{\nu}\to E\to E_{c}\to c^*E_{\nu}
\end{equation}
is distinguished. By \cite{eu}, $E_c$ is an obstruction theory. By the octahedron axiom we obtain that the triangle
\begin{equation}
\epsilon^*L_{\mu}\to E_c\to F\to \epsilon^*L_{\mu}[1]
\end{equation}
is distinguished. From the long exact sequence in cohomology and the fact that $h^{-2}(F)=0$ we obtain that $E_c$ is a perfect obstruction theory for the morphism $c$. Triangle (\ref{thetrueone}) shows that we can apply theorem \ref{connected} to the composition of morphisms $$\bar{M}_{g,n}(\G(1,n),d)\stackrel{^c}\rightarrow\bar{Q}_{g,n}(\G(1,n),d)\stackrel{^{\nu}}\rightarrow\mathfrak{M}.$$ 
\\As $c$ is surjective and $\bar{M}_{g,m}(\G(1,n),d)$ is connected, we get that  $\bar{Q}_{g,m}(\G(1,n),d)$ is connected. The claim now follows from theorem \ref{connected}.
\end{proof}
\begin{remark} The proofs of propositions \ref{projbdl} and \ref{mapquot}  show that whenever we have a commutative diagram
\begin{equation*}
\xymatrix{F\ar[r]^c\ar[d]^{\epsilon}&G\ar[d]^{\nu}\\
\mathfrak{M}_1\ar[r]^{\mu}&\mathfrak{M}_2}
\end{equation*}
where
\begin{enumerate}
\item the bottom row is a morphism of smooth stacks
\item the vertical arrows have (relative) perfect obstruction theories $E_{\epsilon}$, $E_{\nu}$
\item we have a morphism $c^* E_{\nu}\to E_{\epsilon}$ such that its cone is a perfect complex
\end{enumerate} 
then, Theorem \ref{connected} applies to this more general picture.
\end{remark}

Humboldt-Universit\"at zu Berlin
\\Institut f\"ur Mathematik
\\manolach@mathematik.hu-berlin.de

\end{document}